\documentclass[review]{elsarticle}
\usepackage{amsmath,amsthm,amsfonts,amssymb,amscd, amsxtra,mathrsfs, color,subfigure}
\usepackage{lineno,hyperref}
\usepackage{url}
\usepackage[english]{babel}
\usepackage{graphicx}
\usepackage{hyperref}
\usepackage{array}
\usepackage{multirow}
\usepackage{booktabs}
\modulolinenumbers[5]

\journal{Journal of \LaTeX\ Templates}

\newtheorem{theorem}{Theorem}
\newtheorem{lemma}{Lemma}

\newtheorem{remark}{Remark}

\newcommand{\m}{\mathbb{M}}
\newcommand{\tpm}{T_{p}\mathbb{M}}

\DeclareMathOperator{\grad}{grad}

\begin{document}

\begin{frontmatter}

\title{Damped Newton's Method on Riemannian Manifolds}

\author[bort-teles]{Marcio Ant\^onio de A. Bortoloti}
\ead{mbortoloti@uedb.edu.br}
\author[bort-teles]{Teles A. Fernandes\corref{mycorrespondingauthor}}
\ead{telesfernades@uesb.edu.br}
\author{Orizon Pereira Ferreira\fnref{orizonaddress}}
\ead{orizon@ufg.br}
\author{Yuan Jin Yun\fnref{yuanaddress}}
\ead{jin@ufpr.br}




\cortext[mycorrespondingauthor]{Corresponding author}
\ead{support@elsevier.com}

\address[bort-teles]{DCET/UESB, CP-95, CEP 45083-900-Vit\'oria da Conquista, Bahia, Brazil }
\address[orizonaddress]{IME/UFG, CP-131, CEP 74001-970 - Goaiania, Goi\'as, Brazil}
\address[yuanaddress]{DM/CP/UFPR, Jardin das Am\'ericas, CEP 81531-980 - Curitiba, Paran\'a, Brazil}
\begin{abstract}
A {\it  damped Newton's method}  to find a singularity of a vector field in Riemannian  setting is presented  with global convergence study.  It is  ensured that the sequence generated by the proposed method reduces to a sequence generated by the Riemannian version of the classical Newton's method after a finite number of iterations,  consequently its  convergence rate  is superlinear/quadratic. Moreover, numerical experiments illustrate that the damped Newton's method  has  better performance  than the Newton's method in number of iteration and computational time. 
\end{abstract}


\end{frontmatter}


\section{Introduction}
In the 1990s, the optimization on manifolds area  gained considerable popularity, especially with the work of Edelman et al. \cite{EdelmanAriasSmith1999}. Recent years have witnessed a growing interest in the development of numerical algorithms for nonlinear manifolds, as there are many numerical problems posed in manifolds arising in various natural contexts, see  \cite{EdelmanAriasSmith1999, Smith1994, Absil2009}.  Thus, algorithms using the differential structure of nonlinear manifolds play an important role in optimization; see \cite{Absil2014,HuangAndGallivanAndAbsil2015,LiLopezMartin-Marquez2009,LiWang2008,HuLiYang2016,Ring2012,Manton2015,Hosseini2017,GrohsandHosseini2016,ZhaoBaiJin2015,CambierAbsil2016,SatoIwai2013}. In this paper, instead of focusing on finding singularities of gradient vector fields, including local minimizers, on Riemannian manifolds, we consider the more general problem of finding singularities of vector fields. 

Among the methods to find a zero of a nonlinear function, Newton's method under suitable conditions has local quadratic/superlinear convergence rate; see \cite{Dennis1996,Ortega2000}. This remarkable property has motivated several studies on generalizing Newton's method from a linear setting to the Riemannian one \cite{Absil2009,Li2009,Schulz2014,Wang2009,Ferreira2012,Ferreira2002,Li2006,FernandesAndFerreiraAndYuan2017}. Although of Newton's method shows fast local convergence, it is very sensitive with respect to the initial iterate and may diverge if it is not sufficiently close to the solution. To overcome this drawback, some strategies used in Newton's method for optimization problems  were  introduced including BFGS, Levenberg-Marquardt and trust region etc, see \cite{Dennis1996}  and  \cite{Bertsekas2014}.  On the other hand, it is well-known in the linear context that one way to globalize the convergence of Newton's method is to damp its step-size (see \cite{Dennis1996,Solodov2014,Bertsekas2014,Burdakov1980}).  Among the strategies used,  one particularly interesting is  the one by using  a linear search together with  a merit function.  In this case, the basic idea  is to use linear search to damp Newton step-size  when the full step does not provide a sufficient decrease for values of the chosen merit function, which measures the quality of approximation to a zero of the nonlinear function in consideration. Newton's method with these  globalization strategies are called {\it Damped Newton's Methods}. For a comprehensive study about these  subject on linear setting  see, for example, \cite{Ortega2000,Dennis1996,Deuflhard2011}.  In this paper, we generalize this  strategy of  globalization of Newton's method to solve the problem of finding singularities of vector fields defined on Riemannian manifolds. Until now, studies on the globalization strategies in Riemannian settings have been restricted to optimization problems, for example,  the Newton's method with the  hessian of the objective function  updated by \textit{BFGS} family, \cite{HuangAndGallivanAndAbsil2015, Ring2012}, the \textit{Trust-Region} methods, \cite{AbsilAndBakerAndGallivan2007} and \textit{Levenberg-Marquardt} methods \cite[Chapter~8, Section~8.4.2]{Absil2009}.  To the best of our knowledge, a global analysis of the damped Newton's method for finding singularities of vector fields defined on Riemannian manifolds by using  a linear search together with  a merit function, whose particular case is to find local minimizers of real-valued functions defined on Riemannian manifolds,  is novel here. Based on the idea presented in \cite[Section 4]{FacchineiAndKanzow1997} for nonlinear complementarity problem, we propose a damped Newton method in the Riemannian setting. Moreover, we shall show its global convergence  to a singularity of the vector field preserving the same convergence rate of the classical Newton's method. We perform some numerical experiments  for minimizing  families of functions defined on cone of symmetric positive definite matrices which is one of Riemannian manifolds.  Our experiments illustrate numerical performance of the proposed damped Newton method by  linear search   decreasing  a merit function. The numerical results display that the damped Newton improves the behavior of the method when compared to the full step Newton method. 

The remainder this paper is organized as follows. In Section~\ref{sec:basic} we present the notations  and basic results used in the rest paper. In Section~\ref{NewtonGlobal} we describe the global superlinear and quadratic convergence analysis of the damped Newton method. In Section~\ref{sec:NE} we display numerical experiments to verify the main theoretical results. Finally,   we concludes the paper Section~\ref{sec:conclusions}.

\section{Basic Concepts  and Auxiliary Results}\label{sec:basic}
In this section we recall some notations, definitions and auxiliary results of Riemannian geometry  used throughout the paper.  Some basic concepts used here can be found in many introductory books on Riemannian geometry, for example, in \cite{doCarmo1992} and \cite{Sakai1996}.  Let $\mathbb{M}$ be a finite dimensional Riemannian manifold, denote the {\it tangent space} of $\m$ at $p$ by $\tpm$ and the {\it tangent bundle} of $\m$ by $T\m=\bigcup_{p\in\m}\tpm$.  The corresponding norm associated to the Riemannian metric $\langle \cdot ~, ~ \cdot \rangle$ is denoted by $\|  ~\cdot~ \|$. The Riemannian  distance  between $p$ and $q$   in $\mathbb{M}$ is denoted  by $d(p,q)$,  which induces the original topology on $\mathbb{M}$, namely,  $(\mathbb{M}, d)$ that is a complete metric space.   The open ball of radius $r>0$ centred at $p$ is defined as  $B_{r}(p):=\left\lbrace q\in\m:d(p,q)<r\right\rbrace$.  Let  $\Omega \subseteq   \m$ be an open set and denote by ${\cal X}(\Omega)$ the {\it space of  differentiable  vector fields} on $\Omega$.  Let $\nabla$ be the Levi-Civita connection associated to $(\mathbb{M}, \langle \cdot ~, ~ \cdot \rangle)$.  The covariant derivative of $X \in {\cal X}(\Omega)$ determined by $\nabla$ defines at each $p\in \Omega$ a linear map $\nabla X(p):\tpm\to\tpm$ given by $\nabla X(p)v:=\nabla_{Y}X(p),$
where $Y$ is a vector field such that $Y(p)=v$. For $f: \mathbb{M} \to \mathbb{R}$,  a twice-differentiable function the Riemannian metric induces the mappings   $f\mapsto  \mbox{grad} f $ and   $f\mapsto \mbox{Hess} f$, which associate its {\it gradient} and {\it hessian} by the rules 
\begin{equation} \label{eq:GradHess}
\langle \mbox{grad} f,X\rangle:=d f(X),  \qquad    \langle \mbox{Hess} f X,X\rangle :=d^2 f(X, X),  \qquad \forall~ X \in {\cal X}(\Omega), 
\end{equation}
respectively, and  the last equalities imply that  $\mbox{Hess} f X= \nabla_{X}  \mbox{grad} f$,   for all  $X \in {\cal X}(\Omega)$. 
For  each $p\in \Omega$, the { \it conjugate}  of a linear map $A_p:T_p \mathbb{M} \to T_p \mathbb{M}$  is a linear  map  $A^{*}_p:T_p \mathbb{M} \to T_p \mathbb{M}$ defined by $\langle A_p v, u\rangle= \langle v,  A^{*}_pu\rangle$, for all $u, v\in T_{p}M$. The {\it norm }  of the linear map $A_p$  is defined by $\|A_p\|:=\sup \left\{ \|A_pv \|~:~ v\in T_p \mathbb{M}, \,\|v\|=1 \right\}$.  A vector field $V$ along a differentiable curve $\gamma$ in $M$ is said to be {\it parallel} if and only if $\nabla_{\gamma^{\prime}} V=0$. If $\gamma^{\prime}$ itself is parallel we say that $\gamma$ is a {\it geodesic}.  The restriction of a geodesic to a  closed bounded interval is called a {\it geodesic segment}. A geodesic segment joining $p$ to $q$ in $\mathbb{M}$ is said to be {\it minimal} if its length is equal to $d(p,q)$. If there exists a  unique geodesic segment  joining $p$ to $q$, then we denote it  by $\gamma_{pq}$.  For each $t \in [a,b]$, $\nabla$ induces an isometric mapping, relative to $ \langle \cdot , \cdot \rangle  $, $P_{\gamma,a,t} \colon T _{\gamma(a)} {\mathbb{M}} \to T _{\gamma(t)} {\mathbb{M}}$ defined by $ P_{\gamma,a,t}\, v = V(t)$, where $V$ is the unique vector field on $\gamma$ such that $\nabla_{\gamma'(t)}V(t) = 0$, and $V(a)=v$. This mapping is the so-called {\it parallel transport} along  the geodesic segment   $\gamma$ joining  $\gamma(a)$ to $\gamma(t)$. Note also that $ P_{\gamma,\,b_{1},\,b_{2}}\circ P_{\gamma,\,a,\,b_{1}}=P_{\gamma,\,a,\,b_{2}}$ and  $P_{\gamma,\,b,\,a}=P^{-1}_{\gamma,\,a,\,b}$.  {\it When there is no confusion we will  consider the notation $P_{pq}$  instead of $P_{\gamma,\,a,\,b}$ in the case when $\gamma$ is the unique geodesic segment joining $p$ and $q$}. A Riemannian manifold is {\it complete} if the geodesics are defined for any values of $t\in \mathbb{R}$. Hopf-Rinow's theorem asserts that every pair of points in a  complete Riemannian manifold $\mathbb{M}$ can be joined by a (not necessarily unique) minimal geodesic segment.  Due to  the completeness of the Riemannian manifold $\mathbb{M}$, the {\it exponential map} $\exp_{p}:T_{p}\mathbb{M} \to \mathbb{M} $ can be  given by $\exp_{p}v\,=\, \gamma(1)$, for each $p\in \mathbb{M}$, where $\gamma$ is the geodesic defined by its position $p$ and velocity $v$ at $p$. Let $p\in \mathbb{M}$, the {\it injectivity radius} of $\mathbb{M}$ at $p$ is defined by
$$
i_{p}:=\sup\left\{ r>0~:~{\exp_{p}}_{\lvert_{B_{r}(0_{p})}} \mbox{ is\, a\, diffeomorphism} \right\},
$$
where $B_r(0_{p}):=\left\lbrace  v\in T_{p}\mathbb{M}:\parallel v-0_{p}\parallel <r\right\rbrace$ and $0_{p}$ denotes the origin of the tangent plane  $T_{p}\mathbb{M}$. 
\begin{remark}\label{unicidadedageodesica}
	Let $\bar{p}\in \mathbb{M}$. The above definition implies  that if $0<\delta<i_{\bar{p}}$, then $\exp_{\bar p}B_{\delta}(0_{ \bar p})=B_{\delta}( \bar p)$. Moreover, for all $p\in B_{\delta}(\bar p)$,  there exists a unique geodesic segment $\gamma$ joining  $p$ to $\bar p$, which is given by $\gamma_{p\bar p}(t)=\exp_{p}(t \exp^{-1}_{p} {\bar p})$, for all $t\in [0, 1]$.
\end{remark}

Consider  $p\in\mathbb{M}$ and $\delta_{p}:=\min\{1, i_{p}\}$. The quantity assigned to measure how fast the geodesic spread apart in $\m$ has been defined in \cite{Dedieu2003} as 
\begin{multline}\nonumber
K_{p}:=\sup\bigg{\{}  \dfrac{d(\exp_{q}u, \exp_{q}v)}{\parallel u-v\parallel} ~:~ q\in B_{\delta_{p}}(p), ~ u,\,v\in T_{q}\mathbb{M}, ~u\neq v,\\
 ~\parallel v\parallel\leq \delta_{p},
\parallel u-v\parallel\leq \delta_{p}\bigg{\}}.
\end{multline}
\begin{remark} In particular, when $u=0$ or more generally when $u$ and $v$ are on the same line through $0$, $d(\exp_{q}u, \exp_{q}v)=\parallel u-v\parallel$. Hence,  $K_{p}\geq1$,  for all $p\in\m$. Moreover, when $\m$ has non-negative sectional curvature, the geodesics spread apart less than the rays \cite[chapter 5]{doCarmo1992}, i.e., $d(\exp_{p}u, \exp_{p}v)\leq\| u-v\|$ and, in this case, $K_{p}=1$  for all $p\in\m$.
\end{remark}
Let  $X\in {\cal X}(\Omega)$   and $  \bar{p}\in \Omega$. Assume that  $0<\delta<\delta_{\bar{p}}$. Since  $\exp_{\bar{p}}B_{\delta}(0_{\bar{p}})=B_{\delta}(\bar{p})$, there exists a unique geodesic joining  each $p\in B_{\delta}(\bar{p})$ to $\bar{p}$. Moreover, using \cite[ equality 2.3]{Ferreira2002} we obtain
\begin{equation}\label{diferenciabilidade}
X(p)= P_{\bar{p}p}X(\bar{p})+P_{\bar{p}p}\nabla X(\bar{p})\exp^{-1}_{\bar{p}}p+ d(p, \bar{p})r(p),  \qquad \underset{p\to \bar{p}}{\lim}r(p)=0.
\end{equation}
Let $p_{*}\in\mathbb{M}$. The following result establish that, if $\nabla X(p_{*})$ is nonsingular  there exists a neighborhood of $p_{*}$ which $\nabla X$ is nonsingular; see \cite[Lemma 3.2]{FernandesAndFerreiraAndYuan2017}.
\begin{lemma}\label{Le:NonSing}
	Let  $X\in {\cal X}(\Omega)$  and ${\bar p}\in\mathbb{M}$. Assume that $\nabla X$ is  continuous at ${\bar p}$, and  $\nabla X({\bar p})$ is  nonsingular. Then, there exists $0<\hat{\delta}<\delta_{{\bar p}}$ such that $B_{\hat{\delta}}({\bar p})\subset \Omega$, and $\nabla X(p)$ is nonsingular  for each $p\in B_{\hat{\delta}}({\bar p})$.
\end{lemma}
For any   $X\in {\cal X}(\Omega)$ and $p \in\mathbb{M}$ such that $\nabla X(p)$ is  nonsingular, the \textit{Newton's iterate mapping} $N_{X}$  at $p$ is defined  by 
\begin{equation}\label{NewtonIterateMapping}
N_{X}(p):=\exp_{p}(-\nabla X(p)^{-1}X(p)).
\end{equation}
In the folowing we  present a result about the behavior of the Newton's iterate mapping near a singularity of $X$, whose proof can be found in  \cite[Lemma~3.3]{FernandesAndFerreiraAndYuan2017}.
\begin{lemma}\label{L:contNx}
	Let $X:\Omega\to T\mathbb{M}$ a differentiable vector field and $p_{*}\in\mathbb{M}$. Assume  that $X(p_{*})=0$, $\nabla X$ is continuous at $p_{*}$ and $\nabla X(p_{*})$ is  nonsingular.  Then,   $\lim_{p\to p_{*}}[dN_{X}(p), p_{*})/d(p, p_{*})]=0.$
\end{lemma}
The next result establishes  superlinear convergence of the Newton's method; see \cite[Theorem 3.1] {FernandesAndFerreiraAndYuan2017}.
\begin{theorem}\label{superlinearclassicalnewtononmanifold}
	Let $\Omega\subset\mathbb{M}$ be an open set, $X:\Omega\to T\mathbb{M}$ a differentiable vector field and $p_{*}\in\Omega$. Suppose that $p_{*}$ is a singularity of $X$, $\nabla X$ is continuous at $p_{*}$, and $\nabla X(p_{*})$ is  nonsingular. Then,  there exists  $\bar{\delta}>0$ such that,  for  all  $p_{0}\in B_{\bar{\delta}}(p_{*})$, the Newton sequence  $p_{k+1}=\exp_{p_{k}}(-\nabla X(p_{k})^{-1}X(p_{k}))$,  for all  $k=0, 1, \ldots $, 
	is well-defined, contained in $B_{\bar{\delta}}(p_{*})$, and converges superlinearly to $p_{*}$.
\end{theorem}
Let  $X\in {\cal X}(\Omega)$   and $\bar{p}\in \Omega$. The map $\nabla X$ is \textit{locally Lipschitz continuous at} $\bar p$ if and only if there exist a number $L>0$  and a $0<\delta_{L}<\delta_{p^{*}}$ such that $\parallel P_{pp^{*}}\nabla X(p) -\nabla X(p_{*})P_{pp_{*}}\parallel\leq L d(p,p_{*})$, for all  $p\in B_{\delta_{L}}(p_{*}).$ We end this section with a result which, up to some minor adjustments, merges into  \cite[Theorem 4.4]{Smith1994}, see also \cite[Theorem 7.1]{Ferreira2012}. 

\begin{theorem} \label{th:HV1}
	Let $X:{\Omega}\to T\mathbb{M}$ be a continuously differentiable
	vector field, and $p_{*}\in\mathbb{M}$ a singularity of $X$. Suppose that $\nabla X$ is locally Lipschitz continuous at $p_*$ with constant $L>0$, and $\nabla X (p_*)$ is nonsingular. Then, there exists a $r>0$ such that  the sequence $
	p_{k+1} =\exp_{p_k}\left(- \nabla X ({p_k}) ^{-1} X(p_k)\right)$,  for all $k=0,1,\ldots$,  starting  in $p_0\in B_{ r}(p_*)\setminus\{p_*\}$, is well-defined,  contained $B_{ r}(p_*)$, converges to $p_{*}$, and there holds
	$$
	d(p_*,p_{k+1})
	\leq\frac{L\,K_{p_*}^{2}\parallel\nabla X(p^{*})^{-1}\parallel}{2\left[K_{p_*}-d(p_{0},p_{*})L\|\nabla X(p^{*})^{-1}\|\right]}\,d(p_*,p_{k})^{2}, \qquad k=0,\,1,\ldots\, .
	$$
\end{theorem}
\section{Damped Newton Method}\label{NewtonGlobal}
In this section we consider the problem of  finding  a singularity of a differentiable vector field. The formal statement of the problem is as follows:   Find  a point $p\in \mathbb{M}$ such that
\begin{equation} \label{eq:TheProblem}
X(p)=0, 
\end{equation}
where  $X: \Omega \subseteq   \mathbb{M}\to T\mathbb{M}$ is a differentiable vector field.  The  Newton's method applied  to this problem has {\it local} superlinear  convergence,  whenever the covariant derivative in the singularity  of $X$ is non-singular, see \cite{FernandesAndFerreiraAndYuan2017}.  Additionally, if  the covariant derivative is Lipschitz continous at singularity, then the method has $Q$-quadratic convergence, see \cite{Ferreira2012}. In order to globalize the Newton's method keeping the same properties aforementioned, we will use the same idea of the Euclidean
setting by introducing a linear search strategy. This  modification  of Newton's method  called   {\it  damped Newton's method}  performs  a   linear search   decreasing   the   merit function  $\varphi:\mathbb{M}\to\mathbb{R}$, 
\begin{equation}\label{eq:naturalfuncaomerito}
\varphi(p)=\dfrac{1}{2}\parallel X(p)\parallel^{2}, 
\end{equation}
in order to reach the  superlinear convergence region of the Newton's method. In the following we present the  formal description of the damped Newton's algorithm.

\vspace{0.4cm}
\hrule
\noindent
\\
\noindent
\centerline{\bf Damped Newton Method}
\\
\hrule
\begin{description}
	\item[Step 0.] Choose a scalar $\sigma\in(0,1/2)$, take an initial point $p_0 \in \mathbb{M}$, and set $k=0$;
	\item[Step 1.] Compute {\it search direction} $v_{k}\in T_{p_{k}}\mathbb{M}$ as a solution of the linear equation
	\begin{equation}\label{EQNEWTON12}
	X(p_{k})+\nabla X(p_{k})v=0.
	\end{equation}
	\hspace{.6cm }If $v_{k}$ exists, go to {\bf Step $2$}. Otherwise, set the search direction $v_{k}= -\grad\,\varphi(p_{k})$,  where $\varphi$ is defined by \eqref{eq:naturalfuncaomerito}, i.e.,  
	\begin{equation}\label{DirectionOfGradiente}
	v_{k}=-\nabla X(p_{k})^{*}X(p_{k}).
	\end{equation}
	\hspace{.5cm} If $v_{k}=0$, stop;
	\item[Step 2.] Compute the {\it stepsize} by the  rule 
	\begin{multline}\label{eq:ArmijoCond123456}
	\alpha_k:=\max \big{\{} 2^{-j}~: ~\varphi\left(\exp_{p_k}(2^{-j} v_{k})\right)\leq \varphi(p_k)+\\
	\sigma 2^{-j}\left\langle \grad\,\varphi(p_k), ~ v_{k} \right\rangle, ~j\in \mathbb{N} \big{\}};
	\end{multline}
	\hspace{.6cm} and set the {\it next iterated}  as 
	\begin{equation} \label{eq:NM12}
	p_{k+1}:=\exp_{p_{k}}(\alpha_{k}v_{k});
	\end{equation}
	\item[Step 3.] Set $k\leftarrow k+1$ and go to {\bf Step 1}.
\end{description}
\hrule

\noindent
\vspace{0.005cm} 

We remark that in {\bf Step 1}  we  resort  directly to the steepest descent step of $\varphi$   as a safeguard, but only in those cases when the Newtonian direction in \eqref{EQNEWTON12} does not exist. This idea has been used, in Euclidian context, for nonlinear complementarity problems,  see for example  \cite{FacchineiAndKanzow1997} and \cite[Chap.~5, Sec.~5.1.1]{Solodov2014}. 
\subsection{Preliminaries}
In this subsection we present some preliminaries results in order to ensure well-definition and convergence of the sequence generated by the damped Newton's method. We begin  with an useful result for  establishing the well-defined sequence.
\begin{lemma}\label{L:DescDir}
	Let $p\in \Omega$ such that  $X(p)\neq 0$.  Assume that   $v=-\nabla X(p)^{*}X(p)$ or  that  $v$   is a solution of the following equation 
	\begin{equation}\label{eq: GeneralDiretion}
	X(p)+\nabla X(p)v=0.
	\end{equation}
	If $v\neq 0$, then $v$  is a descent direction  for  $\varphi$ from $p$, i.e., $ \left\langle\grad\,\varphi(p),  v\right\rangle <0$ .
\end{lemma}
\begin{proof}
	First we assume that $v$ satisfies  \eqref{eq: GeneralDiretion}.  Since $\grad\,\varphi(p)=\nabla X(p)^{*}X(p)$, there is  $\left\langle\grad\,\varphi(p),  v\right\rangle= \left\langle X(p),  \nabla X(p)v\right\rangle$. It follows from $X(p)\neq 0$ and  $v$ satisfying  \eqref{eq: GeneralDiretion} that  $\left\langle\grad\,\varphi(p),  v\right\rangle= -\parallel X(p)\parallel^{2}<0$,  which implies that $v$  is a descent direction  for  $\varphi$ from $p$. Now, assume that   $v=-\nabla X(p)^{*}X(p)$.  Thus $\left\langle\grad\,\varphi(p),  v\right\rangle=-\|\nabla X(p)^{*}X(p)\|^2<0$ and   $v$ is also a descent direction  for  $\varphi$ from $p$.
\end{proof}
Under suitable assumptions  the  following result  guarantees that the damped Newton's method, after a finite quantity of iterates, reduces to the classical iteration of Newton's method.
\begin{lemma}\label{Le:SupLinConPhi}
	Let  $\bar{p}\in\mathbb{M}$ satisfy $X(\bar p)=0$.  If $\nabla X$ and $\nabla X(\bar p)$  are continuous at $\bar p$ and nonsingular respectively,  then  there exists $0<\hat{\delta}<\delta_{\bar{p}}$ such that $B_{\hat{\delta}}(\bar{p})\subset \Omega$, $\nabla X(p)$ is nonsingular  for each $p\in B_{\hat{\delta}}(\bar{p})$ and 
	\begin{equation}\label{eq:SupelMerit}
	\lim_{p\to \bar p}\frac{\varphi\left(N_{X}(p)\right)}{\parallel X(p)\parallel^{2}}=0.
	\end{equation}
	As a consequence, there exists a $\delta>0$ such that,  for all  $\sigma\in(0,1/2)$ and $\delta< \hat{\delta}$  there holds
	\begin{equation}\label{eq:ArmijoCond}
	\varphi\left(N_{X}(p)\right) \leq \varphi\left(p)\right)+ \sigma \left\langle\grad\,\varphi(p), -\nabla X(p)^{-1}X(p)\right\rangle, \qquad  \forall~p\in B_{\delta}(\bar{p}).
	\end{equation}
\end{lemma}
\begin{proof}
	By  Lemma~\ref{Le:NonSing}  there exists $0<\hat{\delta}<\delta_{\bar{p}}$ such that $B_{\hat{\delta}}(\bar{p})\subset \Omega$ and $\nabla X(p)$ is nonsingular  for each $p\in B_{\hat{\delta}}(\bar{p})$. Define $v_{p}=   -\nabla X(p)^{-1}X(p)$, for all $p\in B_{\hat{\delta}}(\bar{p})$. Since $X(\bar p)=0$  and $\nabla X$ is continuous and nonsingular at $\bar p$,   there is ${\lim_ {p\to \bar{p}}}v_{p}=0$. Moreover,  it follows from  \eqref{diferenciabilidade}, the  isometry of the parallel transport and  $\|\exp^{-1}_{\bar{p}}\exp_{p}(v_{p})\| =d(\exp_{p}(v_{p}), \bar{p})$ that
	\begin{multline}\nonumber
	\varphi\left(N_{X}(p)\right)=  \dfrac{1}{2}\parallel X\left(\exp_{p}(v_{p})\right)-P_{\bar p \exp_{p}(v_{p})}X(\bar p)\parallel^{2}\leq\\
	  \left[\| \nabla X(\bar{p})\| + \| r(\exp_{p}(v_{p}))\|\right]^2 d^2(\exp_{p}(v_{p}), \bar{p}). 
	\end{multline}
	Hence, after some simple  algebraic manipulations and  from the last inequality we obtain for all $p\in B_{\hat{\delta}}(\bar{p})\backslash\{\bar{p}\}$,
	\begin{equation}\label{sssss}
	\dfrac{\varphi\left(\exp_{p}(v_{p})\right)}{\parallel X(p)\parallel^{2}} \leq  \left[\| \nabla X(\bar{p})\| + \| r(\exp_{p}(v_{p}))\|\right]^2 \dfrac{ d^2(\exp_{p}(v_{p}), \bar{p})}{d^{2}(p, \bar p)}\dfrac{d^{2}(p, \bar p)}{\parallel X(p)\parallel^{2}}.
	\end{equation}
	On the other hand,  owing that  $X(\bar p)=0$ and $\nabla X(\bar p)$ is nonsingular, it is easy to see that 
	\begin{multline}\label{eq:AAABBB}
	\exp^{-1}_{\bar p}p=-\nabla X(\bar p)^{-1}\left[P_{p\bar{p}}X(p)- X(\bar p)- \nabla X(\bar p)\exp^{-1}_{\bar p}p\right]+\\
	 \nabla X(\bar p)^{-1}P_{p\bar{p}}X(p).
	\end{multline}
	Using again  \eqref{diferenciabilidade}  and that  $\nabla X(\bar p)$ is nonsingular, we conclude that there exists $0<\bar{\delta}<\delta_{\bar{p}}$ such that 
	$$
	\left\|\nabla X(\bar p)^{-1}\left[P_{p\bar{p}}X(p)- X(\bar p)- \nabla X(\bar p)\exp^{-1}_{\bar p}p\right]\right\| \leq \frac{1}{2}    d(p, \bar{p}), \qquad  \forall ~p\in B_{\bar{\delta}}(\bar{p}).
	$$
	Thus, combining \eqref{eq:AAABBB} with last inequality and considering  that $d(\bar{p}, p)=\|\exp^{-1}_{\bar p}p\|$ we obtain that  $d(\bar{p}, p)\leq   d(p, \bar{p})/2  + \|\nabla X(\bar p)^{-1}P_{p\bar{p}}X(p)\|$, for all $p\in B_{\bar{\delta}}(\bar{p})$, which implies that  $d^2(\bar{p}, p)\leq  \left[2 \|\nabla X(\bar p)^{-1}\|\right]^2\|X(p)\|^2$, for all $p\in B_{\bar{\delta}}(\bar{p})$. Hence, letting $ \tilde{\delta}=\min\{ \hat{\delta}, \bar {\delta} \}$ we obtain from \eqref{sssss} that for all $p\in B_{\tilde{\delta}}(\bar{p})\backslash\{\bar{p}\}$, 
	$$
	\dfrac{\varphi\left(\exp_{p}(v_{p})\right)}{\parallel X(p)\parallel^{2}} \leq  \left[2\|\nabla X(\bar{p})^{-1}\|\left( \| \nabla X(\bar{p})\| + \| r(\exp_{p}(v_{p}))\|\right)\right]^2 \dfrac{ d^2(\exp_{p}(v_{p}), \bar{p})}{d^{2}(p, \bar p)}.  
	$$
	Therefore, it follows from Lemma~\ref{L:contNx} and  $\lim_{p\to \bar{p}}r(\exp_{p}(v_{p}))=0$ that the equality  \eqref{eq:SupelMerit} by taking limit, as $p$ goes to $\bar{p}$,  in the latter  inequality.  For proving \eqref{eq:ArmijoCond}, we first use \eqref{eq:SupelMerit} for concluding that  there exists a $\delta>0$  such that,  $\delta< \hat{\delta}$ and for   $\sigma \in (0,1/2)$ there is
	$$
	\varphi\left(N_{X}(p)\right) \leq  \dfrac{1-2\sigma}{2}\parallel X(p)\parallel^{2},\quad \forall~ p\in B_{\delta}(\bar p).
	$$
	Taking into account  that $\grad\,\varphi(p)= \nabla X(p)^{*}X(p)$ we can conclude directly that $\left\langle\grad\,\varphi(p),  -\nabla X(p)^{-1}X(p)\right\rangle= -\| X(p)\|^{2}$, then the  last inequality  is equivalent to \eqref{eq:ArmijoCond} and then  the proof is complete.
\end{proof}
In the next result we show that whenever the vector field  is continuous with  nonsingular  covariant derivative at a singularity, there exists a neighborhood  which  is  invariant by the Newton's iterate mapping associated.
\begin{lemma}\label{L:InvNx}
	Let  $\bar{p}\in\mathbb{M}$ such that $X(\bar p)=0$. If $\nabla X$ is continuous at $\bar p$ and $\nabla X(\bar{p})$ is nonsingular,  then  there exists $0<\hat{\delta}<\delta_{\bar{p}}$ such that $B_{\hat{\delta}}(\bar{p})\subset \Omega$ and $\nabla X(p)$ is nonsingular  for each $p\in B_{\hat{\delta}}(\bar{p})$. Moreover,    $N_{X}(p)\in B_{\hat{\delta}}(\bar{p})$, for all $p\in B_{\hat{\delta}}(\bar{p})$.
\end{lemma}
\begin{proof}
	It follows from Lemma~\ref{Le:NonSing}  that there exists $0<\hat{\delta}<\delta_{\bar{p}}$ such that $B_{\hat{\delta}}(\bar{p})\subset \Omega$ and  $\nabla X(p)$ is nonsingular  for each $p\in B_{\hat{\delta}}(\bar{p})$. Thus,  shirking $\hat{\delta}$  if necessary,  we can use Lemma~\ref{L:contNx}  to conclude  that $d(N_{X}(p), \bar{p})<d(p, \bar{p})/2$, for all $p\in B_{\hat{\delta}}(p_{*})$, which implies  the last statement of the lemma.
\end{proof}
\subsection{Convergence Analysis}
In this subsection we  establish our  main result.  We shall prove a result on the    global convergence of  the damped Newton's method  preserving the fast convergence  rates  of the  Newton's one.  We begin by proving the well-definition of the sequence generated by the  damped Newton's method.
\begin{lemma}\label{boadef_Damped1}
	The  sequence $\{p_k\}$ generated by the  damped Newton's method  is well-defined.
\end{lemma}
\begin{proof}
	If $v_0\neq 0$ and $X(p_0)\neq 0$, then by using Lemma~\ref{L:DescDir} we can conclude that   $v_0$ satisfying     \eqref{EQNEWTON12} or  \eqref{DirectionOfGradiente}   is a descent direction  for  $\varphi$ from $p_0$. Hence, by using ordinary argument we conclude that $\alpha_{0}$  in \eqref{eq:ArmijoCond123456} is well-defined, and $p_{1}$ in \eqref{eq:NM12} is also well-defined. Therefore, by using induction argument  we can prove that the sequence   $\{p_k\}$  is well-defined.
\end{proof}
Note  that,   if the sequence generated by the damped Newton's method is finite, then the last iterate is a solution of \eqref{eq:TheProblem} or  a critical point of  $\varphi$ defined in \eqref{eq:naturalfuncaomerito}. Thus, we can assume that  $\{p_k\}$  is infinite,    $v_k\neq 0$ and $X(p_k)\neq 0$,  for all $k=0, 1, \ldots$

\begin{theorem}\label{theoremAuxiliary123}
	Let $X: \Omega \subseteq   \mathbb{M}\to T\mathbb{M}$ be a differentiable vector field.  Assume that   $\left\{p_{k}\right\}$ generated by  the Damped Newton's Method has  an accumulation point   $\bar p\in \Omega$ such that $\nabla X$ is  continuous at $\bar p$ with nonsingular $\nabla X(\bar p)$. Then, $\left\{p_{k}\right\}$  converges superlinearly to the singularity $\bar p$ of $X$.  Moreover, the  convergence rate is quadratic provided that $\nabla X$ is locally Lipschitz continuous at $\bar p$.
\end{theorem}
\begin{proof}
	We begin to show that $\bar p$  is a singularity of $X$.  As aforementioned, we can assume that  $\left\{p_{k}\right\}$ is infinite,  $\grad\,\varphi(p_{k})\neq0$ and $X(p_k)\neq 0$, for all $k=0,\,1,\,\ldots$. Hence,  \eqref{eq:ArmijoCond123456} and \eqref{eq:NM12} imply that
	\begin{multline}\nonumber
	\varphi(p_{k})-\varphi(p_{k+1})\geq-\sigma\alpha_{k}\langle\grad\,\varphi(p_{k}), v_{k}\rangle\\
	 =\begin{cases}
	\sigma\alpha_{k} \| X(p_k)\|^{2}>0, & \text{if} ~v_{k} ~\text{satisfies} ~\eqref{EQNEWTON12};  \\
	\sigma\alpha_{k} \|\grad\,\varphi(p_{k})\|^{2}>0, & \text{else}.
	\end{cases}
	\end{multline}
	Then  $\left\{\varphi(p_{k})\right\}$ is   strictly decreasing and, bounded below by zero which results in convergence of the sequence. Hence, taking the limite  in this latter inequality we obtain
	$
	\lim_{k\to \infty}  [\alpha_{k}\left\langle \grad\,\varphi(p_{k}), v_{k}\right\rangle]~=~0.
	$
	Let  $\{p_{k_j}\}$ be a subsequence of  $\{p_{k}\}$  such that  $\lim_{k_j\to +\infty}p_{k_j}=\bar p$. Thus, we have
	\begin{equation}\label{ma2}
	\lim_{k_j\to \infty}  \left[\alpha_{k_j}\left\langle \grad\,\varphi(p_{k_j}), v_{k_j}\right\rangle \right]= 0.
	\end{equation}
	Owing that  $\nabla X$ is continuous at $\bar p$ and  $\nabla X(\bar p)$ is nonsingular,  by using  Lemma~\ref{Le:NonSing},  we can also assume that   $\nabla X(p_{k_{j}})$  is nonsingular, for all $j=0, 1, \ldots$. Hence,  \eqref{EQNEWTON12} has a  solution and then for $j=0, 1, \ldots$ we have
	\begin{equation} \label{eq:retwo}
	v_{k_{j}}=-\nabla X(p_{k_{j}})^{-1}X(p_{k_{j}}),\qquad \left\langle \grad\,\varphi(p_{k_j}), v_{k_j}\right\rangle=-\|X(p_{k_{j}})\|^{2}.
	\end{equation}
	To analyse the consequences of  \eqref{ma2}  we consider the following  two possible cases:
	$$
	{\bf a)}~ {\lim\inf}_{j\to \infty} \alpha_{k_{j}}>0,  \qquad   \qquad  \qquad  {\bf b)} ~{\lim\inf}_{j\to \infty} \alpha_{k_{j}}=0.
	$$
	First we assume item a) holds. From \eqref{ma2}, passing onto a further subsequence if necessary, we can assume that 
	\begin{equation}\label{ksks12}
	\lim_{j\to \infty} \left\langle \grad\,\varphi(p_{k_j}), v_{k_j}\right\rangle= 0.
	\end{equation}
	Taking the limit in the latter equality,  as $j$ goes to infinity,  and  considering that $\lim_{j\to +\infty}p_{k_j}=\bar p$ and  $X$ is continuous at $\bar p$, we conclude from  \eqref{ksks12} and \eqref{eq:retwo} that $X(\bar p)=0$ what means $\bar p$  is a singularity of $X$. Now, we assume item b) holds. Hence, given $s\in\mathbb{N}$ we can take $j$ large enough such that $\alpha_{k_{j}}<2^{-s}$. Thus   $2^{-s}$ does not satisfies  the Armijo's condition \eqref{eq:ArmijoCond123456}, i.e., 
	$$
	\varphi\left(\exp_{p_{k_{j}}}(2^{-s} v_{k})\right)> \varphi(p_{k_{j}})+\sigma 2^{-s}\left\langle \grad\,\varphi(p_{k_{j}}), ~ v_{{k_{j}}} \right\rangle.
	$$
	Letting $j$ goes to infinity,   considering that the exponential  mapping is continuous,   $\lim_{j\to +\infty}p_{k_j}=\bar p$ and due to $\nabla X$ and $X$ be   continuous at $\bar p$,    it follows from  \eqref{eq:retwo} and the  last inequality that 
	$$
	\varphi\left(\exp_{\bar p}(2^{-s} {\bar v})\right)> \varphi(\bar p)+\sigma 2^{-s}\left\langle \grad\,\varphi(\bar p), ~ {\bar v} \right\rangle, 
	$$
	where $ {\bar v}=   -\nabla X(\bar p)^{-1}X(\bar p) $. Hence, we obtain $ [\varphi\left(\exp_{\bar p}(2^{-s} {\bar v})\right)- \varphi(\bar p)]/2^{-s}\geq \sigma \left\langle \grad\,\varphi(\bar p), ~ {\bar v} \right\rangle.$ Therefore, letting $s$  goes to infinity  we can conclude that $\left\langle \grad\,\varphi(\bar p), ~ {\bar v} \right\rangle  \geq \sigma \left\langle \grad\,\varphi(\bar p), ~ {\bar v} \right\rangle$, or equivalently
	$ \|X(\bar{p})\|^{2}\leq \sigma\|X(\bar{p}\|^{2}$. Thus,   owing $\sigma\in(0,1/2)$ we have  $\parallel X(\bar p)\parallel=0$ which implies that   $\bar p$  is singularity of $X$.
	
	We proceed to prove that   there exists  $k_{0}$ such that $\alpha_{k}=1$,  for all $k\geq k_{0}$.   Since   $\nabla X(\bar p)$ is nonsingular and  $X(\bar{p})=0$,  the  Lemma~\ref{Le:NonSing} and Lemma~\ref{L:InvNx} imply that there exists   $\hat{\delta}>0$ such that $\nabla X(p)$ is nonsingular  and $N_{X}(p)\in B_{{\delta}}(\bar p)$,  for all $p\in B_{{\delta}}(\bar p)$ and all  $\delta\leq \hat{\delta}$. Since  $\bar p$ is   a cluster  point of    $\left\{p_{k}\right\}$, there exits $k_{0}$ such that $p_{k_{0}} \in B_{\hat{\delta}}(\bar p)$ (shirking $\hat{\delta}$ if necessary). It follows from Lemma \ref{Le:SupLinConPhi} that 
	$
	\varphi\left(N_{X}(p_{k_{0}})\right) \leq \varphi\left(p_{k_{0}})\right)+\sigma\left\langle\grad\,\varphi(p_{k_{0}}), v_{k_{0}}\right\rangle, 
	$
	with $v_{k_{0}}=-\nabla X(p_{k_{0}})^{-1}X(p_{k_{0}})$. Hence, it follows from the last inequality,    \eqref{NewtonIterateMapping} and  \eqref{eq:ArmijoCond123456} that  $\alpha_{k_{0}}=1$. By  \eqref{eq:NM12}  we can conclude that $p_{k_{0}+1}= N_{X}(p_{k_{0}})$ which implies $ p_{k_{0}+1}\in B_{\hat{\delta}}(\bar p)$ because of $N_{X}(p_{k_{0}}) \in B_{\hat{\delta}}(\bar p)$. Then,  an  induction step is completely analogous,   yielding 
	\begin{equation}\label{indicesdentrodabaciarapida123}
	\alpha_{k}=1, \qquad    p_{k+1}= N_{X}(p_{k}) \in B_{\hat{\delta}}(\bar p), \qquad \forall ~k\geq k_{0}.
	\end{equation}
	In order to obtain supetlinear convergence of the entire sequence $\left\{p_{k}\right\}$, let $\bar\delta >0$ be  given by  Theorem~\ref{superlinearclassicalnewtononmanifold}. Thus,  making $\hat{\delta}$ smaller if necessary so that $\hat{\delta}<\bar\delta$,  we can apply Theorem~\ref{superlinearclassicalnewtononmanifold} to conclude from \eqref {indicesdentrodabaciarapida123}  that  $\left\{p_{k}\right\}$ converges  superlinearly to  $\bar p$.
	
To prove that $\left\{p_{k}\right\}$ converges quadratically to $\bar p$,  we first  make $\hat{\delta}<r$, where $r$ is given by  Theorem~\ref{th:HV1}. Since $\nabla X$ is locally Lipschitz continuous at $\bar{p}$,   the result follows from the combination of   \eqref{indicesdentrodabaciarapida123} and Theorem \ref{th:HV1}.
\end{proof}
\section{Numerical Experiments}\label{sec:NE}
In this section,  we shall give some numerical experiments  to illustrate the performance of the damped Newton method for minimizing  two families  of functions defined on the cone of symmetric positive definite matrices. A particular instance of one of these families has application in robotics. Before present our numerical experiment, we need to introduce some concepts. Let ${\mathbb P}^{n}$ be the set of  symmetric matrices of order $n\times n$ and ${\mathbb P}^{n}_{++}$ be the cone of  symmetric positive definite matrices.  Following Rothaus \cite{Rothaus1960}, let $\mathbb{M}:=({\mathbb P}^n_{++}, \langle \, , \, \rangle)$ be the Riemannian manifold endowed with the Riemannian metric defined  by  
\begin{equation}\label{eq:metric}
\langle U,V \rangle=\mbox{tr} (V\psi''(P)U)=\mbox{tr} (VP^{-1}UP^{-1}),\, P\in\mathbb{M}, \quad U,V\in
T_P\mathbb{M} \approx~\mathbb{P}^n,
\end{equation}
where $\mbox{tr}P$ denotes the trace of  $P$. Then, the exponential mapping is given by
\begin{equation}\label{Exp:1}
\exp_PV=P^{1/2}e^{\left(P^{-1/2}VP^{-1/2}\right)}P^{1/2}, \qquad P\in\mathbb{M}, \qquad V\in
T_P\mathbb{M} \approx~\mathbb{P}^n.
\end{equation}
Let $X$ and $Y$ be vector fields in $\mathbb{M}$. Then, by using \cite[Theorem 1.2, page 28]{Sakai1996},  we can prove that the  Levi-Civita connection of $\mathbb{M}$  is given by 
\begin{equation} \label{eq:dercov}
\nabla_Y X(P)=X'Y-\frac{1}{2}\left[ YP^{-1}X+ XP^{-1}Y \right], 
\end{equation}
where $P\in \mathbb{M}$, and $ X'$ denotes  the  Euclidean derivative of $X$. Therefore, it follows  from \eqref{eq:metric} and  \eqref{eq:dercov} that the Riemannian gradient and hessian of a  twice-differentiable function $f: \mathbb{M} \to {\mathbb R}$ are respectively  given by:
\begin{multline} \label{eq:GradHess}
\mbox{grad} f(P)=Pf'(P)P,\qquad  \mbox{Hess}f(P)V=Pf''(P)VP+\\
\frac{1}{2}\left[  Vf'(P)P +  PF'(P)V \right],
\end{multline}
for all $V\in T_P\mathbb{M}$,  where  $f'(P)$ and  $f''(P)$ are the  Euclidean gradient and  hessian of $f$ at $P$, respectively. In this case, by using \eqref{Exp:1},  the Newton's iteration  for finding $P\in \mathbb{M}$ such that  $\mbox{grad} f(P)=0$  is given by 
\begin{equation}\label{MNewtonSPD}
P_{k+1}=P_{k}^{1/2}e^{P_{k}^{-1/2}V_{k}P_{k}^{-1/2}}P_{k}^{1/2},  \qquad k=0, 1, \ldots,
\end{equation}
where, by using again the  equalities in  \eqref{eq:GradHess}, $V_k$ is the unique solution of the Newton's linear equation
\begin{equation}\label{Newton_no_ConeSPD}
P_{k}f''(P_{k})V_{k}P_{k}+\frac{1}{2}\left[  V_{k}f'(P_{k})P_{k} +  P_{k}f'(P_{k})V_{k} \right]=-P_{k}f'(P_{k})P_{k}, 
\end{equation}
which corresponds to \eqref{EQNEWTON12} for $X= \mbox{grad} f$. Now, we are going to present  concrete  examples for \eqref{MNewtonSPD} and  \eqref{Newton_no_ConeSPD}.  Let  $i=1, 2$ and $f_{i}:\mathbb{P}^{n}_{++}\to\mathbb{R}$ be  defined, respectively,  by 
\begin{equation}\label{eq:familyfuction}
f_1(P)=\mbox{a}_{1}\ln\,\det P + \mbox{b}_{1}\,\mbox{tr}\,P^{-1}, \quad \quad f_2(P)=\mbox{a}_{2}\ln\det P - \mbox{b}_{2}\mbox{tr}\,P,
\end{equation}
where $\mbox{b}_{1}/\mbox{a}_{1}>0$, $\mbox{a}_{2}/\mbox{b}_{2}>0$, and $\det P$  denotes the determinant  of  $P$, respectively.  In particular, the  function $f_2$ can be associated to the optimal dextrous hand grasping problem in robotics, see  \cite{robotic2002,robotic2007,robotic2002_2}.

Then, for each  $i=1, 2$,   the Euclidean gradient and hessian of  $f_{i}$ are given, respectively, by
\begin{multline}\label{eq:qwjfvqek1}
f_1'(P)=\mbox{a}_{1}P^{-1}-\mbox{b}_{1}P^{-2},\\
f_1''(P)V=\mbox{b}_{1}\left(P^{-1}VP^{-2}+P^{-2}VP^{-1}\right)-
\mbox{a}_{1}P^{-1}VP^{-1}, 
\end{multline}
\begin{multline}\label{eq:qwjfvqek2}
f_2'(P)=\mbox{a}_{2}P^{-1}-\mbox{b}_{2}I, \qquad \qquad  \quad  f_2''(P)V_{k}=-\mbox{a}_{2}P^{-1}VP^{-1}, 
\end{multline}
where $I$ denotes the  $n  \times n $ identity matrix.  It follows from \eqref{eq:GradHess}, \eqref{eq:familyfuction}, \eqref{eq:qwjfvqek1} and \eqref{eq:qwjfvqek2} that
\begin{equation}\label{eq:gradfamily}
\mbox{grad}\,f_1(P)=\mbox{a}_{1}P-\mbox{b}_{1}I, \quad \qquad \mbox{grad}\,f_2(P)=\mbox{a}_{2}P-\mbox{b}_{2}P^{2}.
\end{equation}
From the last two equalities, we can conclude that the global minimizer of $f_i$,  for each  $i=1, 2$ are $P_1^{*}=\mbox{b}_{1}/\mbox{a}_{1}\,I$ and $P_2^{*}=\mbox{a}_{2}/\mbox{b}_{2}\,I$, respectively.  Our task is to execute explicitly the Damped Newton method to find the global minimizer of $f_i$,  for each  $i=1, 2$. For this purpose, consider $\varphi_{i}:\mathbb{P}_{++}^{n}\to\mathbb{R}$ given by $\varphi_{i}(P)=1/2\left\|\mbox{grad}\,f_{i}(P)\right\|^{2}$, $i=1, 2$. Thus, by using \eqref{eq:gradfamily}  we conclude that
\begin{equation}\label{eq:gradmeritfunctionfamily}
\grad\,\varphi_{1}(P)=\mbox{ab}I -\mbox{b}_{1}^{2}P^{-1}, \qquad  \qquad \grad\,\varphi_{2}(P)=d^{2}P^{3}-\mbox{a}_{2}\mbox{b}_{2}P^{2}.
\end{equation}
By combining  \eqref{eq:metric}, \eqref{Exp:1} and \eqref{eq:gradfamily},  after some calculations  we  obtain  the following equalities 
\begin{align}
\varphi_{1}\left(\exp_{P} V\right)&=\dfrac{1}{2} \mbox{tr}\left(\mbox{a}_{1}I-\mbox{b}_{1}\left(P^{1/2}e^{P^{-1/2}VP^{-1/2}}P^{1/2}\right)^{-1}\right)^{2}, \qquad\label{eq:ancqcwjhqwh1}\\
\varphi_{2}\left(\exp_{P}V\right)&=\dfrac{1}{2} \mbox{tr}\left(\mbox{a}_{2}I-\mbox{b}_{2}P^{1/2}e^{P^{-1/2}VP^{-1/2}}P^{1/2}\right)^{2}. \qquad\label{eq:ancqcwjhqwh2}
\end{align}
Finally,  by using  \eqref{MNewtonSPD}, \eqref{Newton_no_ConeSPD}, definition of $f_{1}$ in \eqref{eq:familyfuction}, \eqref{eq:qwjfvqek1}, left hand  sides of \eqref{eq:gradfamily} and \eqref{eq:gradmeritfunctionfamily} and \eqref{eq:ancqcwjhqwh1} we give  the damped Newton method, for finding the global minimizer of $f_1$.  The formal algorithm to find  the global minimizer of function $f_2$ will not be presented here because it can be obtained similarly  by using  \eqref{MNewtonSPD}, \eqref{Newton_no_ConeSPD}, the definition of $f_{2}$ in \eqref{eq:familyfuction}, \eqref{eq:qwjfvqek2}, right hand sides of \eqref{eq:gradfamily} and \eqref{eq:gradmeritfunctionfamily}  and \eqref{eq:ancqcwjhqwh2}. The Damped Newton algorithim for the matrix cone is given as follows.  \\

\vspace{0.4cm}
\hrule
\noindent
\\
\noindent
\centerline{\bf Damped Newton Method in $\mathbb{P}^{n}_{++}$  (DNM-$\mathbb{P}^{n}_{++}$)}
\\
\hrule
\begin{description}
	\item[Step 0.] Choose a scalar $\sigma\in(0,1/2)$, take an initial point $P_0 \in \mathbb{M}$, and set $k=0$;
	\item[Step 1.] Compute {\it search direction} $V_{k}$, as a solution of the  linear equation
	\begin{align} \label{eq:eqNewtonRopensimetric}
	P_{k}V_{k}+V_{k}P_{k}&=2(P_{k}^2-a/bP_{k}^3).
	\end{align}
	If $V_{k}$ exists go to {\bf Step $2$}. Othewise, put
	 $V_{k}=-\grad\,\varphi_{1}(p_{k})$, i.e, 
	\begin{align} \label{DirectionOfGradientecone}
	V_{k}=\mbox{b}_{1}^{2}P_{k}^{-1}-\mbox{ab}I.
	\end{align}
	\hspace{.5cm} If $V_{k}=0$, stop;
	\item[Step 2.] Compute the {\it stepsize} by the  rule
	\begin{multline}\label{eq:ArmijoCond1234561}
	\alpha_k:=\max\bigg{\{}2^{-j}:\dfrac{1}{2}\emph{tr}\left(\mbox{a}_{1}I-\mbox{b}_{1}\left(P_{k}^{1/2}e^{2^{-j}P_{k}^{-1/2}VP_{k}^{-1/2}}P_{k}^{1/2}\right)^{-1}\right)^{2}\leq\\
	\left(\dfrac{1}{2}-\sigma2^{-j}\right) \emph{tr}\left(\mbox{a}_{1}I-\mbox{b}_{1}P_{k}^{-1}\right)^{2}\bigg{\}}
	\end{multline}
	and set the {\it next iterated}  as 
	\begin{equation} \label{eq:NM121}
	P_{k+1}:=P_{k}^{1/2}e^{\alpha_{k}P_{k}^{-1/2}V_{k}P_{k}^{-1/2}}P_{k}^{1/2};
	\end{equation}
	\item[Step 3.] Set $k\leftarrow k+1$ and go to {\bf Step 1}.
\end{description}
\hrule
\noindent
\\

Although the domain of $f_1$ is a subset of the symmetric matrix set, namely, $\mathbb{P}^{n}_{++}$, equality \eqref{eq:NM12} shows that   DNM-$\mathbb{P}^{n}_{++}$ generates only feasible points without using projections or any other procedure to remain the feasibility.  Hence, problem of minimizing $f_1$ onto   $\mathbb{P}^{n}_{++}$ can be seen as unconstrained Riemannian optimization problem.  Note that, in this case,   the equation \eqref{eq:eqNewtonRopensimetric}  always has  unique solution $V_{k}\in T_{P_{k}}\mathbb{P}^{n}_{++}$, see \cite[Th. 8.2.1]{datta2004}. Consequently the direction $V_{k}$ in \eqref{DirectionOfGradientecone} does not play any role here.  Thus, in this case, we can compare the Newton and damped Newton methods in number of iterations, function evaluation and CPU time to reach the solution. The  Newton method is formally described as follow.\\

\vspace{0.3cm}
\hrule
\noindent
\\
\centerline{\bf   Newton Method  in $\mathbb{P}^{n}_{++}$ (NM-$\mathbb{P}^{n}_{++}$)}
\\
\hrule
\begin{description}
	\item[Step 0.]  Take an initial point $P_0 \in \mathbb{M}$, and set $k=0$;
	\item[Step 1.] Compute {\it search direction} $V_{k}$ as a solution of the linear equation
	$$
	P_{k}V_{k}+V_{k}P_{k}=2(P_{k}^2-\mbox{a}_{1}/\mbox{b}_{1}P_{k}^3).
	$$
	\item[Step 2.]  Compute the {\it next iterated} by 
	$$
	P_{k+1}:=P_{k}^{1/2}e^{P_{k}^{-1/2}V_{k}P_{k}^{-1/2}}P_{k}^{1/2};
	$$
	\item[Step 3.] Set $k\leftarrow k+1$ and go to {\bf Step 1}.
\end{description}
\hrule
\noindent
\vspace{0.005cm}

We have implemented the above two algorithms  by using MATLAB R2015b.  The 
experiments are performed on an Intel Core 2 Duo Processor 2.26 GHz , 4 GB of RAM, and OSX operating system. We compare the iterative behavior of  NM-$\mathbb{P}^{n}_{++}$ and DNM-$\mathbb{P}^{n}_{++}$ applied to $f_1$ and $f_2$ with randomly chosen initial guesses. In all experiments the stop criterion at the iterate $P_{k}\in\mathbb{P}^{n}_{++}$  is  $\left\|\mbox{grad}\,f\left(P_{k}\right)\right\|\leq10^{-8}$ where $\left\|\cdot\right\|$ is norm associated  to the metric given by \eqref{eq:metric}.  All codes are freely available at \url{http://www2.uesb.br/professor/mbortoloti/wp-content/uploads/2018/03/MATLAB_codes.zip}.
\begin{figure}[ht]
	\centering
	\subfigure[ $f_{1}$ with $\mbox{b}_{1}/\mbox{a}_{1}=0.1$ and $n=1000$.]{
		\includegraphics[scale=0.6]{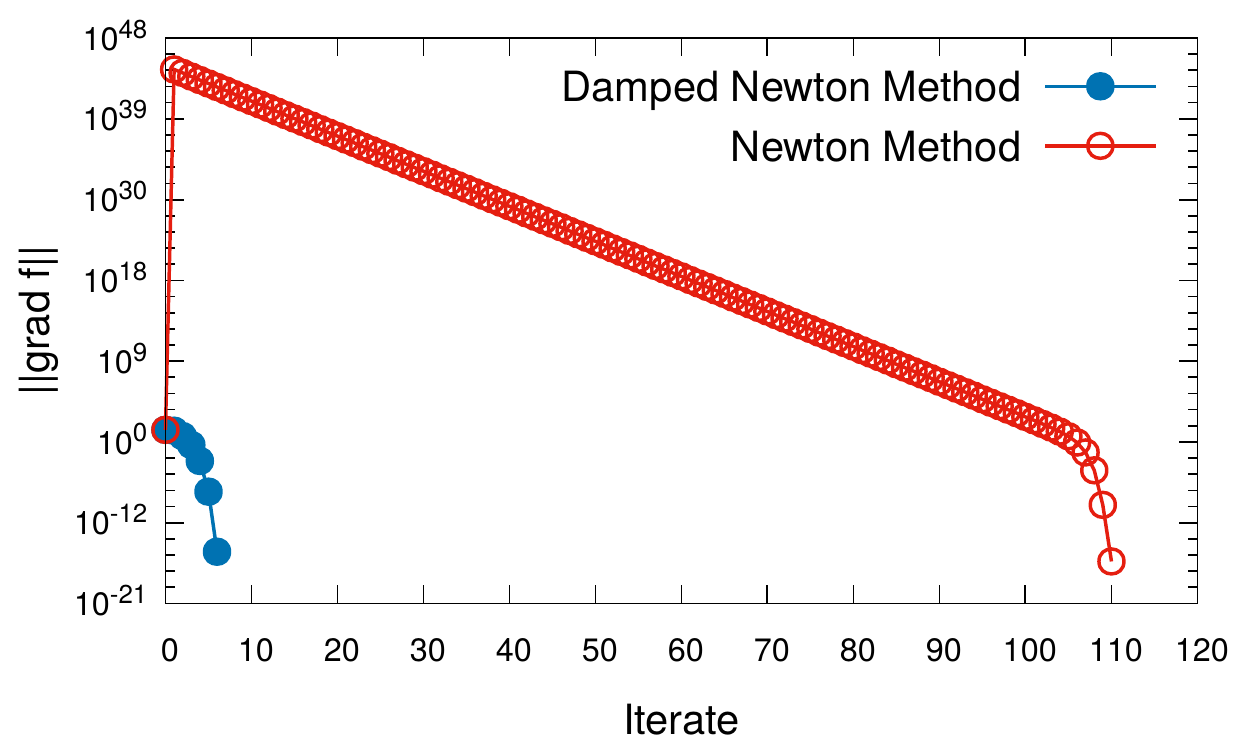}
		\label{fig:pfa}
	}
	\subfigure[$f_{2}$ with $\mbox{b}_{2}/\mbox{a}_{2}=0.002$ and $n=1000$.]{
		\includegraphics[scale=0.6]{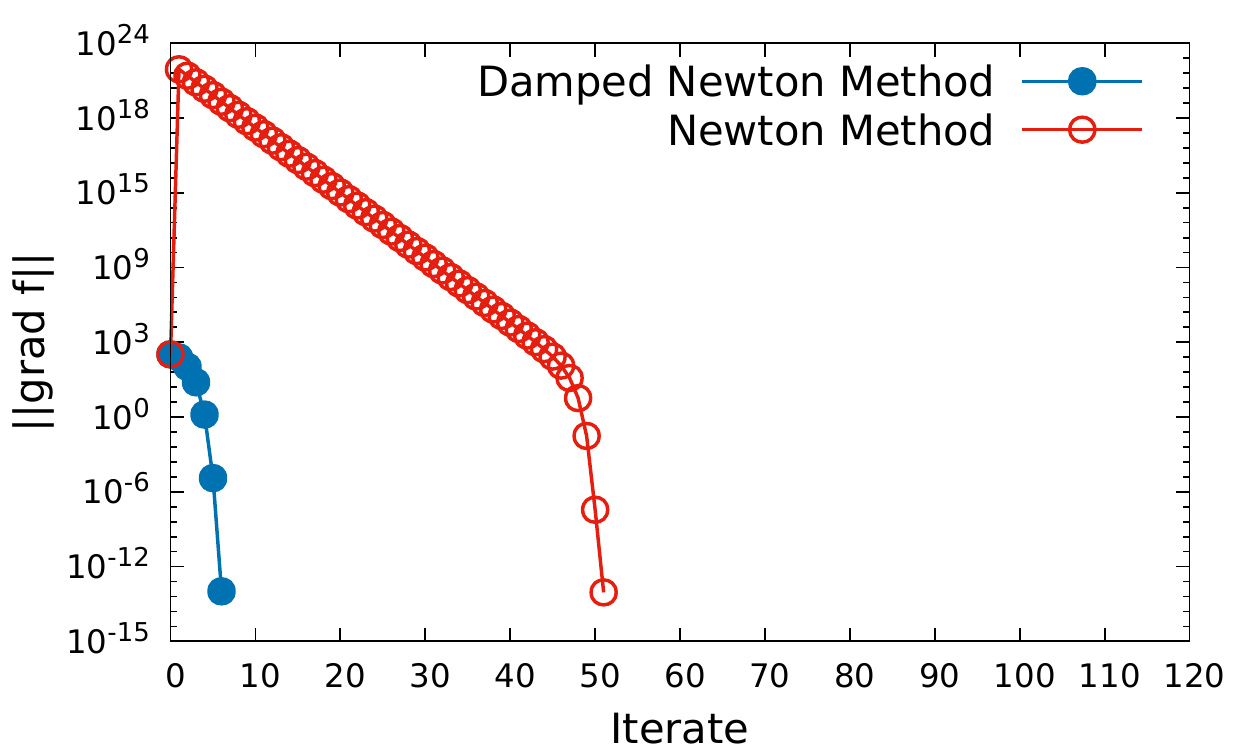}
		\label{fig:pfb}
	}
	\label{figpf}
\end{figure}
We run the above two algorithms in a huge number of  problems by varying the dimensions and parameters $\mbox{b}_{1}/\mbox{a}_{1}$ and $\mbox{b}_{2}/\mbox{a}_{2}$ in functions $f_{1}$ and  $f_{2}$, respectively. In general, DNM-$\mathbb{P}^{n}_{++}$ is  superior to NM-$\mathbb{P}^{n}_{++}$ in number of iterations and CPU time.   As  we can see in  figures (a) and (b)    NM-$\mathbb{P}^{n}_{++}$  performs huge number of  iterations before reaching the superlinear convergence region while the line-search of DNM-$\mathbb{P}^{n}_{++}$ decreasing gradient norm ensuring a superior performance in number of iterations. Note that in these figures  DNM-$\mathbb{P}^{n}_{++}$ and NM-$\mathbb{P}^{n}_{++}$ have the same behavior  when the iterations close to the solution,  which is the consequence of Lemma~\ref{Le:SupLinConPhi}. 
In part the efficiency of the DNM-$\mathbb{P}^{n}_{++}$ can be explained due to the  linearsearch decreasing the merit function.

Table $1$  shows that DNM-$\mathbb{P}^{n}_{++}$ is superior to  NM-$\mathbb{P}^{n}_{++}$ in number of iterations and CPU time, besides showing the number of evaluation of the function in the  linear-search to compute the stepsize. In the columns of this table we read: $n$ is the dimension of $\mathbb{P}^{n}$, $\mbox{b}_{1}/\mbox{a}_{1}$ and $\mbox{b}_{2}/\mbox{a}_{2}$ are the parameters  defining the functions $f_1$ and $f_2$, respectively, NIT is the number of iterates to reach the minimaizer of the function, HE is the number of Hessian evaluation, GE is the number of Gradient evaluation and T is the CPU time in seconds. It is important to observe GE in DNM-$\mathbb{P}^{n}_{++}$ takes into account the evaluations in the Armijo's rule also. We know that, in general,  the linear-search to compute the stepsize is expensive, but the computational effort did not grow significantly with the dimension of of $\mathbb{P}^{n}$ in the DNM-$\mathbb{P}^{n}_{++}$. In each method we can observe that each iteration evaluates the Hessian once. Hence, DNM-$\mathbb{P}^{n}_{++}$ has a superior performance than NM-$\mathbb{P}^{n}_{++}$ because the minimizer is achieved by DNM-$\mathbb{P}^{n}_{++}$ with less iterates than NM-$\mathbb{P}^{n}_{++}$, as can be seen in Table 1.
\begin{table}\label{tab:f1}
	\centering
	{\small
		\begin{tabular}{p{0.3cm}p{0.5cm}p{1.5cm}p{0.4cm}p{0.4cm}p{0.4cm}p{2.0cm}p{0.4cm}p{0.4cm}p{0.4cm}p{0.8cm}}
			\toprule
			& & & \multicolumn{4}{c}{\hspace{-1.0cm}{\bf NM-$\mathbb{P}^{n}_{++}$}} & \multicolumn{4}{c}{\hspace{-0.6cm}{\bf DNM-$\mathbb{P}^{n}_{++}$}}\\
			&$b_{i}/a_{i}$ & $n$ & NIT & HE & GE & T & NIT & HE & GE &  T\\
			\midrule
			\multirow{9}{10em}{$f_{1}$}&\multirow{3}{4em}{$0.1$}  & $1$   &   $41$ & $41$ & $41$ & $0.19$ & $4$& $4$ & $21$  & $0.10$ \\
&& $100$    &   $109$&  $109$  & $109$   & $1.90$ & $6$ & $6$  & $30$ &  $0.59$\\
&& $1000$  &   $110$& $110$  & $110$ & $1240$ & $6$&  6 & 30  &  $107$\\\cline{2-11}
			
&\multirow{3}{4em}{$1.0$}  & $1$   & $7$ & $7$ & $7$ & $0.10$ & $4$& $4$  & 17 &$0.11$ \\
&& $100$    &   $13$  & $13$ &  $13$  & $0.44$ & $4$ & $4$ & 18 &  $0.43$\\
&& $1000$    &   $13$ &  $13$ &  $13$  &  $147$ & $4$ & $4$  & 18 &  $64.90$\\\cline{2-11}
			
&\multirow{3}{4em}{$1.5$}  & $1$   & $6$ & $6$  & $6$  & $0.10$ & $4$ & $4$  & 17 & $0.12$ \\
&& $100$&$10$& $10$ & $10$ &$0.22$&$5$& $5$ &22&  $0.25$\\
&& $1000$    &   $10$ & $10$  & $10$ & $113$ & $6$ & $6$  & 27 &   $97.10$\\\midrule

\multirow{9}{10em}{$f_{2}$}&\multirow{3}{4em}{$0.001$} &$1$&$249$&$249$&$249$&$0.37$&$6$&$6$ &31&$0.17$ \\
&& $100$&$100$&$100$&$100$&$1.66$&$7$&$7$&34&$0.42$\\
&&$1000$&$100$& $100$  & $100$  &$1130.00$&$7$&$7$&34&$131.00$\\
\cline{2-11}
&\multirow{3}{4em}{$0.002$} &$1$&$125$&$125$& $125$&$0.29$&$6$& $6$ &30&$0.10$\\
&&$100$&$51$&$51$&$51$ &$0.90$&$6$&$6$&29&$0.27$\\
&& $1000$&$51$&$51$&$51$& $627$&$6$&$6$&29&$109$\\
\cline{2-11}
&\multirow{3}{4em}{$0.01$} & $1$   &   $26$ & $26$ & $26$ &  $0.31$ & $5$ & $5$ & 23  &   $0.16$\\
&& $100$&$12$& $12$ &$12$  & $0.92$&$5$&$5$  &22&  $0.21$\\
&& $1000$&$12$& $12$ & $12$ & $138$&$5$&$5$&22&$85.00$\\
\toprule
\end{tabular}
\caption{\small Performance of DNM-$\mathbb{P}^{n}_{++}$ and  NM-$\mathbb{P}^{n}_{++}$ to $f_{1}$ and $f_{2}$ functions.}}
\end{table}

\section{Final Remarks}\label{sec:conclusions}
In this paper, in order to find a singularity of a vector field defined on Riemannian manifold, we presented a damped Newton's method  and established its  global convergence with quadratic/superlinear convergence rate.  Note that the global convergence analysis  of the damped Newton's method without assuming nonsingularity  of the hessian of the objective function at its critical points was established in the Euclidean context; see \cite[Chapter~1, Section~1.3.3]{Bertsekas2014}. Since many important problems in the context of a Riemannian manifold become the problem of finding a singularity of a vector field \cite{Adler2002,EdelmanAriasSmith1999}, it would be interesting to obtain a similar global analysis for a damped Newton's method to find a singularity of a vector field defined on a Riemannian manifold.  Based on our previous experience, we feel that more tools need to be developed to achieve this goal. In order to obtain a better numerical efficiency to the  damped Newton's method our next task is to design new methods where the equation  \eqref{EQNEWTON12} is approximately solved, besides using  retraction in \eqref{eq:NM12} and more efficient linear seach in  \eqref{eq:ArmijoCond123456}.


\begin{thebibliography}{10}
\expandafter\ifx\csname url\endcsname\relax
  \def\url#1{\texttt{#1}}\fi
\expandafter\ifx\csname urlprefix\endcsname\relax\def\urlprefix{URL }\fi
\expandafter\ifx\csname href\endcsname\relax
  \def\href#1#2{#2} \def\path#1{#1}\fi

\bibitem{EdelmanAriasSmith1999}
A.~Edelman, T.~A. Arias, S.~T. Smith, The geometry of algorithms with
  orthogonality constraints, SIAM J. Matrix Anal. Appl. 20~(2) (1999) 303--353.

\bibitem{Smith1994}
S.~T. Smith, Optimization techniques on {R}iemannian manifolds, in: Hamiltonian
  and gradient flows, algorithms and control, Vol.~3 of Fields Inst. Commun.,
  Amer. Math. Soc., Providence, RI, 1994, pp. 113--136.

\bibitem{Absil2009}
P.-A. Absil, R.~Mahony, R.~Sepulchre, Optimization algorithms on matrix
  manifolds, Princeton University Press, Princeton, NJ, 2008.

\bibitem{Absil2014}
P.-A. Absil, L.~Amodei, G.~Meyer, Two {N}ewton methods on the manifold of
  fixed-rank matrices endowed with {R}iemannian quotient geometries, Comput.
  Statist. 29~(3-4) (2014) 569--590.

\bibitem{HuangAndGallivanAndAbsil2015}
W.~Huang, K.~A. Gallivan, P.-A. Absil, A {B}royden class of quasi-{N}ewton
  methods for {R}iemannian optimization, SIAM J. Optim. 25~(3) (2015)
  1660--1685.

\bibitem{LiLopezMartin-Marquez2009}
C.~Li, G.~L{\'o}pez, V.~Mart{\'{\i}}n-M{\'a}rquez, Monotone vector fields and
  the proximal point algorithm on {H}adamard manifolds, J. Lond. Math. Soc. (2)
  79~(3) (2009) 663--683.

\bibitem{LiWang2008}
C.~Li, J.~Wang, Newton's method for sections on {R}iemannian manifolds:
  generalized covariant {$\alpha$}-theory, J. Complexity 24~(3) (2008)
  423--451.

\bibitem{HuLiYang2016}
Y.~Hu, C.~Li, X.~Yang, On convergence rates of linearized proximal algorithms
  for convex composite optimization with applications, SIAM J. Optim. 26~(2)
  (2016) 1207--1235.

\bibitem{Ring2012}
W.~Ring, B.~Wirth, Optimization methods on {R}iemannian manifolds and their
  application to shape space, SIAM J. Optim. 22~(2) (2012) 596--627.

\bibitem{Manton2015}
J.~H. Manton, A framework for generalising the {N}ewton method and other
  iterative methods from {E}uclidean space to manifolds, Numer. Math. 129~(1)
  (2015) 91--125.

\bibitem{Hosseini2017}
S.~Hosseini, A.~Uschmajew, A {R}iemannian gradient sampling algorithm for
  nonsmooth optimization on manifolds, SIAM J. Optim. 27~(1) (2017) 173--189.

\bibitem{GrohsandHosseini2016}
P.~Grohs, S.~Hosseini, Nonsmooth trust region algorithms for locally
  {L}ipschitz functions on {R}iemannian manifolds, IMA J. Numer. Anal. 36~(3)
  (2016) 1167--1192.

\bibitem{ZhaoBaiJin2015}
Z.~Zhao, Z.-J. Bai, X.-Q. Jin, A {R}iemannian {N}ewton algorithm for nonlinear
  eigenvalue problems, SIAM J. Matrix Anal. Appl. 36~(2) (2015) 752--774.

\bibitem{CambierAbsil2016}
L.~Cambier, P.-A. Absil, Robust low-rank matrix completion by {R}iemannian
  optimization, SIAM J. Sci. Comput. 38~(5) (2016) S440--S460.

\bibitem{SatoIwai2013}
H.~Sato, T.~Iwai, A {R}iemannian optimization approach to the matrix singular
  value decomposition, SIAM J. Optim. 23~(1) (2013) 188--212.

\bibitem{Dennis1996}
J.~E. Dennis, Jr., R.~B. Schnabel, Numerical methods for unconstrained
  optimization and nonlinear equations, Vol.~16 of Classics in Applied
  Mathematics, Society for Industrial and Applied Mathematics (SIAM),
  Philadelphia, PA, 1996, corrected reprint of the 1983 original.

\bibitem{Ortega2000}
J.~M. Ortega, W.~C. Rheinboldt, Iterative solution of nonlinear equations in
  several variables, Vol.~30 of Classics in Applied Mathematics, Society for
  Industrial and Applied Mathematics (SIAM), Philadelphia, PA, 2000, reprint of
  the 1970 original.

\bibitem{Li2009}
C.~Li, J.-H. Wang, J.-P. Dedieu, Smale's point estimate theory for {N}ewton's
  method on {L}ie groups, J. Complexity 25~(2) (2009) 128--151.

\bibitem{Schulz2014}
V.~H. Schulz, A {R}iemannian view on shape optimization, Found. Comput. Math.
  14~(3) (2014) 483--501.

\bibitem{Wang2009}
J.-H. Wang, C.~Li, Kantorovich's theorems for {N}ewton's method for mappings
  and optimization problems on {L}ie groups, IMA J. Numer. Anal. 31~(1) (2011)
  322--347.

\bibitem{Ferreira2012}
O.~P. Ferreira, R.~C.~M. Silva, Local convergence of {N}ewton's method under a
  majorant condition in {R}iemannian manifolds, IMA J. Numer. Anal. 32~(4)
  (2012) 1696--1713.

\bibitem{Ferreira2002}
O.~P. Ferreira, B.~F. Svaiter, Kantorovich's theorem on {N}ewton's method in
  {R}iemannian manifolds, J. Complexity 18~(1) (2002) 304--329.

\bibitem{Li2006}
C.~Li, J.~Wang, Newton's method on {R}iemannian manifolds: {S}male's point
  estimate theory under the {$\gamma$}-condition, IMA J. Numer. Anal. 26~(2)
  (2006) 228--251.

\bibitem{FernandesAndFerreiraAndYuan2017}
T.~A. Fernandes, O.~P. Ferreira, J.~Yuan, On the {S}uperlinear {C}onvergence of
  {N}ewton's {M}ethod on {R}iemannian {M}anifolds, J. Optim. Theory Appl.
  173~(3) (2017) 828--843.

\bibitem{Bertsekas2014}
D.~P. Bertsekas, Constrained optimization and {L}agrange multiplier methods,
  Computer Science and Applied Mathematics, Academic Press, Inc. [Harcourt
  Brace Jovanovich, Publishers], New York-London, 2014.

\bibitem{Solodov2014}
A.~F. Izmailov, M.~V. Solodov, Newton-type methods for optimization and
  variational problems, Springer Series in Operations Research and Financial
  Engineering, Springer, Cham, 2014.

\bibitem{Burdakov1980}
O.~Burdakov, Some globally convergent modifications of newton's method for
  solving systems of nonlinear equations, in: Soviet mathematics-Doklady,
  Vol.~22, 1980, pp. 376--378.

\bibitem{Deuflhard2011}
P.~Deuflhard, Newton methods for nonlinear problems, Vol.~35 of Springer Series
  in Computational Mathematics, Springer, Heidelberg, 2011, affine invariance
  and adaptive algorithms, First softcover printing of the 2006 corrected
  printing.

\bibitem{AbsilAndBakerAndGallivan2007}
P.-A. Absil, C.~G. Baker, K.~A. Gallivan, Trust-region methods on {R}iemannian
  manifolds, Found. Comput. Math. 7~(3) (2007) 303--330.

\bibitem{FacchineiAndKanzow1997}
F.~Facchinei, C.~Kanzow, A nonsmooth inexact {N}ewton method for the solution
  of large-scale nonlinear complementarity problems, Math. Programming 76~(3,
  Ser. B) (1997) 493--512.

\bibitem{doCarmo1992}
M.~P. do~Carmo, Riemannian geometry, Mathematics: Theory \& Applications,
  Birkh\"auser Boston, Inc., Boston, MA, 1992, translated from the second
  Portuguese edition by Francis Flaherty.

\bibitem{Sakai1996}
T.~Sakai, Riemannian geometry, Vol. 149 of Translations of Mathematical
  Monographs, American Mathematical Society, Providence, RI, 1996, translated
  from the 1992 Japanese original by the author.

\bibitem{Dedieu2003}
J.-P. Dedieu, P.~Priouret, G.~Malajovich, Newton's method on {R}iemannian
  manifolds: convariant alpha theory, IMA J. Numer. Anal. 23~(3) (2003)
  395--419.

\bibitem{Rothaus1960}
O.~S. Rothaus, Domains of positivity, Abh. Math. Sem. Univ. Hamburg 24 (1960)
  189--235.

\bibitem{robotic2002}
U.~Helmke, S.~Ricardo, S.~Yoshizawa, Newton's algorithm in {E}uclidean {J}ordan
  algebras, with applications to robotics, Commun. Inf. Syst. 2~(3) (2002)
  283--297, dedicated to the 60th birthday of John B. Moore, Part I.

\bibitem{robotic2007}
G.~Dirr, U.~Helmke, C.~Lageman, Nonsmooth {R}iemannian optimization with
  applications to sphere packing and grasping, in: Lagrangian and {H}amiltonian
  methods for nonlinear control 2006, Vol. 366 of Lect. Notes Control Inf.
  Sci., Springer, Berlin, 2007, pp. 29--45.

\bibitem{robotic2002_2}
U.~Helmke, K.~Hüper, J.~B. Moore, Quadratically convergent algorithms for
  optimal dextrous hand grasping, IEEE Transactions on Robotics and Automation
  18~(2) (2002) 138--146.

\bibitem{datta2004}
B.~N. Datta, Numerical methods for linear control systems: design and analysis,
  Vol.~1, Academic Press, 2004.

\bibitem{Adler2002}
R.~L. Adler, J.-P. Dedieu, J.~Y. Margulies, M.~Martens, M.~Shub, Newton's
  method on {R}iemannian manifolds and a geometric model for the human spine,
  IMA J. Numer. Anal. 22~(3) (2002) 359--390.

\end{thebibliography}

\end{document}